\documentclass[1 leqno,11pt,psfig]{amsart}
\usepackage{amssymb, amsmath,amsmath,latexsym,amssymb,amsfonts,amsbsy, amsthm,mathtools,graphicx,CJKutf8,CJKnumb,CJKulem,color}
\usepackage{graphicx,epsfig}
\usepackage{graphicx}
\usepackage{amssymb}
\usepackage{graphicx}
\usepackage{graphicx,epsfig}
\usepackage{amsmath}
\usepackage{mathrsfs}
\usepackage{amssymb}

\usepackage{amssymb,mathrsfs,amsfonts,amsmath,amsbsy,tikz}\usepackage{fontenc}\usepackage{textcomp}

\numberwithin{equation}{section}

\allowdisplaybreaks

\newcommand{\beq}{\begin{equation}}
\newcommand{\eeq}{\end{equation}}
\newcommand{\ben}{\begin{eqnarray}}
\newcommand{\een}{\end{eqnarray}}
\newcommand{\beno}{\begin{eqnarray*}}
\newcommand{\eeno}{\end{eqnarray*}}

\newtheorem{theorem}{Theorem}[section]

\newtheorem{lemma}[theorem]{Lemma}

\begin{document}
\begin{CJK*}{UTF8}{gkai}
\title[A note on eigenvalues] {A note on eigenvalues of a class of  singular continuous and discrete linear Hamiltonian systems}

\author{Hao Zhu}
\address{Chern Institute of Mathematics, Nankai University, Tianjin 300071, P. R. China}
\email{haozhu@nankai.edu.cn}

\maketitle

\begin{abstract}
 In this paper,  we show that the analytic and geometric multiplicities of an eigenvalue of a class of singular linear Hamiltonian systems  are equal, where both  endpoints are  in the limit circle cases. The proof is fundamental and is given for both  continuous and discrete   Hamiltonian systems. The  method used in this paper also works for both endpoints are regular, or  one endpoint is regular and the other is in the limit circle case.
\end{abstract}

\section{Introduction}
Consider the following singular continuous linear Hamiltonian system:
\begin{align}\label{singular continuous linear Hamiltonian system}
Jy'(t)=(P(t)+\lambda W(t))y(t),\;t\in(a,b),
\end{align}
where both  endpoints  $a$ and $b$ are  in the limit circle cases,
 $W(t)$ and $P(t)$ are $2m\times2m$   Hermitian matrices for any $t\in(a,b)$, $W,P\in L_{loc}(a,b)$, $W\geq 0$ on $(a,b)$,  and
\[
J=%
\begin{pmatrix}
0 & -I_m\\
I_m & 0
\end{pmatrix}.
\]
Throughout this paper, for the singular  continuous linear Hamiltonian system, we assume that

{\it (A) For each subinterval $(a',b')\subset(a,b)$, if $y$ satisfies $Jy'-Py=W\hat f$ and $Wy=0$ on $(a',b')$ for some $\hat f\in \hat L^2_{W}(a',b')$, then $y=0$ on $(a',b')$,}\\
where $\hat L_{W}^2(a',b')$ is defined in (\ref{pre-space in continuous case}). If (A) holds, then, for all $\lambda\in \mathbb{C}$ and for all nontrivial solutions $y$ of system
(\ref{singular continuous linear Hamiltonian system}), the following definiteness condition can be verified \cite{Atkinson1964}:
\begin{align}\label{definiteness condition conitinuous}
\int_{a'}^{b'}y^*(s)W(s)y(s)ds>0,
\end{align}
where $y^*$ is the complex conjugate transpose of $y$.

Consider the following singular discrete linear Hamiltonian system:
\begin{align}\label{singular discrete linear Hamiltonian system}
J\Delta y(n)=(P(n)+\lambda W(n))R(y)(n),\;n\in\mathbb{Z},
\end{align}
where $\mathbb{Z}$ is the set of integers, $\pm\infty$  are in the limit circle cases,
$\Delta y(n)=y(n+1)-y(n)$ is the forward difference operator,  $R(y)(n)=(u^T(n+1),v^T(n))^T$ is the partial right shift operator  with $y(n)=(u^T(n),v^T(n))^T$ and $u(n),v(n)\in\mathbb{C}^m$,
$W(n)={\rm diag}\{W_1(n),W_2(n)\}$ with $W_i(n)\geq0$, $i=1,2$, and
\[P(n)=\begin{pmatrix}-C(n)&A^*(n)\\A(n)&B(n)\end{pmatrix}\]
 are $2m\times2m$ Hermitian matrices with $m\times m$ matrices $A(n)$, $B(n)$ and $C(n)$.

To ensure the existence and uniqueness of the solution of the initial value problem for (\ref{singular discrete linear Hamiltonian system}), in this paper we assume that

{\it ($B_1$) $I_m-A(n)$ is invertible for any $n\in\mathbb{Z}$.}

In the discrete case, the analogous condition to (A) does not always hold \cite{Ren-Shi2011}. So we just assume the definiteness condition, which is similar to (\ref{definiteness condition conitinuous}), but requires conditions on both half lines to adjust the self-adjoint extension as in  \cite{Ren-Shi2014}:

{\it ($B_2$) There exist two finite sequences $\{n\}_{n=s_0^l}^{t_0^l}$  and $\{n\}_{n=s_0^r}^{t_0^r}$ such that $t_0^l\leq0$, $s_0^r\geq 1$ and for any $\lambda\in\mathbb{C}$, every non-trivial solution $y$ of (\ref{singular discrete linear Hamiltonian system}) satisfies
\begin{align*}
\sum_{n=s_0^l}^{t_0^l}R(y)^*(n)W(n)R(y)(n)>0,\;\;\sum_{n=s_0^r}^{t_0^r}R(y)^*(n)W(n)R(y)(n)>0.
\end{align*}}

The spectrum of the singular Hamiltonian system (\ref{singular continuous linear Hamiltonian system})  (or (\ref{singular discrete linear Hamiltonian system})) with a self-adjoint boundary condition consists of  discrete eigenvalues. These eigenvalues can be regarded as zeros of a specified entire function, and the analytic multiplicity of an eigenvalue is referred to its order as a zero of the function.
It renders lots of information in studying eigenvalue problems. However, people may pay more attention to  geometric multiplicity of the eigenvalue in application because it provides the number of eigenfunctions corresponding to the eigenvalue. So, it is nature to ask which of the two multiplicities is larger.

In the literature,  relationship of  analytic and geometric multiplicities of an eigenvalue of the boundary value problem  has been studied  by many researchers.
 Eastham {\it et al.} showed equivalence of the multiplicities of an eigenvalue of continuous
  Sturm-Liouville problems (briefly, SLPs) with  coupled boundary conditions \cite{Eastham1999}, while  Kong {\it et al.}
showed the same result  for separated boundary conditions \cite{Kong2000}. For the singular case,  Kong {\it et al.} proved this equivalence for continuous SLPs with limit circle  endpoints \cite{Kong2004}. Wang and Wu gave an alternative proof depending on the geometric structure of the space of self-adjoint boundary conditions \cite{Wang2005}. We proved the same relationship for continuous high-dimensional SLPs \cite{Hu-Liu-Wu-Zhu2018}. For higher-order ordinary differential equations, we refer the readers to \cite{Naimark1968,Shi2010}.
For  discrete SLPs,
we proved the equivalence of the multiplicities of an eigenvalue in \cite{Zhu2016}. For  continuous linear Hamiltonian systems or  SLPs, the equivalence of multiplicities of an eigenvalue can also be deduced by using Hill-type formula \cite{Hu-Wang2011,Hu-Wang20161,Hu-Wang20162}.  In this paper,
motivated by \cite{Hu-Liu-Wu-Zhu2018,Naimark1968}, we provide a  fundamental way to show  that the two multiplicities of an eigenvalue   are equal for both singular continuous and discrete  linear Hamiltonian systems with limit circle endpoints.
 Though our framework is based on both endpoints  in the limit circle cases, the  method used in this paper also works for both endpoints are regular, or  one endpoint is regular and the other is in the limit circle case.

The rest of this paper is organized as follows. Self-adjoint boundary conditions and  basic properties  of eigenvalues are given in Section 2. Equalities of  analytic and geometric multiplicities of an eigenvalue of the singular  linear Hamiltonian system in continuous and discrete cases are proved in Sections 3 and 4, respectively.

\section{Self-adjoint boundary conditions and  multiplicity of eigenvalues}

\subsection{Continuous case}
Define the space
\begin{align}\label{pre-space in continuous case}
\hat L_W^2(a,b)=\{\hat f:\int_{a}^b\hat f^*(t)W(t)\hat f(t)dt<\infty\}
\end{align}
 with the semi-inner product
$\langle \hat f,\hat g\rangle_c=\int_a^b\hat g^*(t)W(t)\hat f(t)dt$ and
\begin{align*}
L_W^2(a,b)=\hat L_W^2(a,b)/\{\hat y\in \hat L_W^2(a,b):\|\hat y\|_{c}=0\}.
\end{align*}

The maximal, pre-minimal and minimal operators  corresponding to (\ref{singular continuous linear Hamiltonian system}) are defined as follows:
\begin{align*}D(H)=&\{y\in L^2_W(a,b): \textrm{there exists } \hat y\in y \textrm{ such that } \hat y\in AC_{loc}(a,b) \textrm{ and } \mathcal{L}(\hat y)(t)=W(t)\hat f(t) \\
&\textrm{on }t\in(a,b) \textrm{ for some } f\in L_W^2(a,b) \textrm{ and for  any }\hat f\in f\},\\
Hy=&f,\\
D(H_{00})=&\{y\in D(H): \textrm{there exists } \hat y\in y \textrm{ such that }\hat y \textrm{ has a compact support in }(a,b) \textrm{ and}\\
 &\mathcal{L}(\hat y)(t)=W(t)\hat f(t) \textrm{ on } t\in(a,b) \textrm{ for some }f\in L^2_W(a,b) \textrm{ and for any }\hat f\in f\},\\
H_{00}y=&Hy,\;H_0y=\overline H_{00}y,
\end{align*}
where $AC_{loc}(a,b)$ is the set of functions which are locally absolutely continuous on $(a,b)$, and
$\mathcal{L}(\hat y)(t)=J\hat y'(t)-P(t)\hat y(t).$  We know by (A) that $H$ and $H_{00}$ are well-defined  and for any $y\in D(H)$, there exists a unique $\hat y\in y$ such that
 $\mathcal{L}(\hat y)(t)=W(t)\hat f(t)$ for $t\in (a,b)$. So we briefly identify $y$ and $\hat y$ below when $y\in D(H)$.
Thus
\[(f,g)_c(t):=g^*(t)Jf(t)\]
is well-defined for $f,g\in D(H)$.
It follows from \cite{Sun-Shi2010} that
\begin{align*}
\int_{a'}^{b'}\{g^*(t)\mathcal{L}(f)(t)-\mathcal{L}(g)^*(t)f(t)\}dt=(f,g)_c(t)|_{a'}^{b'}
\end{align*}
for any $a<a'<b'<b$, which yields that $(f,g)_c(a):=\lim_{t\to a^+}(f,g)_c(t)$ and $(f,g)_c(b):=\lim_{t\to b^-}(f,g)_c(t)$ exist and are finite for any $f,g\in D(H)$.

Let (\ref{singular continuous linear Hamiltonian system}) be in the limit circle cases at both endpoints $a$ and $b$.
The definition of limit circle endpoint depends on  defect indices of the left and right minimal operators, and we refer to Definition 3.1 in \cite{Shi2004}. Then we know again from \cite{Shi2004} that (\ref{singular continuous linear Hamiltonian system}) has $2m$ linearly independent solutions in $\hat L^2_W(a,b)$ for every $\lambda\in \mathbb{C}$.
Let $\phi_{i,\lambda}\in \hat L^2_W(a,b), 1\leq i\leq 2m$, be the linearly independent solutions of (\ref{singular continuous linear Hamiltonian system})  satisfying the initial condition:
\begin{align*}
(\phi_{1,\lambda}(c_0),\cdots, \phi_{2m,\lambda}(c_0))=I_{2m},
\end{align*}
where  $c_0$ is a fixed point in $(a,b)$. Denote $\phi_i=\phi_{i,0}$, $1\leq i\leq 2m$. For any $t\in[a,b]$, we define
\begin{align*}
\Phi_{\lambda}(t):=\begin{pmatrix}(\phi_{1,\lambda},\phi_1)_c(t)&\cdots&(\phi_{2m,\lambda},\phi_{1})_c(t)\\
\vdots&&\vdots\\(\phi_{1,\lambda},\phi_{2m})_c(t)&\cdots&(\phi_{2m,\lambda},\phi_{2m})_c(t)
\end{pmatrix}.
\end{align*}
For a fixed $t\in[a,b]$ and any $1\leq i,j\leq 2m$, $(\phi_{i,\lambda},\phi_j)_c(t)$ is an entire function of $\lambda$. Then the self-adjoint extension of $H_0$  is stated as follows.

\begin{lemma} \label{self-adjoint extension continuous case}
Assume that (\ref{singular continuous linear Hamiltonian system}) is in the limit circle cases at both endpoints $a$ and $b$. Then $D\subset D(H)$ is a self-adjoint extension domain of $H_0$ if and only if there exist $2m\times2m$ matrices $M$ and $N$ such that
${\rm rank}(M,N)=2m, MJM^*=NJN^*$, and
\begin{align}\label{boundary condition continuous}
D=\left\{y\in D(H):M\begin{pmatrix}(y,\phi_1)_c(a)\\\vdots\\(y,\phi_{2m})_c(a)\end{pmatrix}-
N\begin{pmatrix}(y,\phi_1)_c(b)\\\vdots\\(y,\phi_{2m})_c(b)\end{pmatrix}=0\right\}.
\end{align}
\end{lemma}
For the proof of Lemma \ref{self-adjoint extension continuous case}, we refer the readers to Theorem 5.6 in \cite{Sun-Shi2010}. The self-adjoint extension of $H_0$ with the domain $D$ is denoted by $H_D$. Let $\Gamma_c(\lambda)=\det(M\Phi_{\lambda}(a)-N\Phi_{\lambda}(b)), \lambda\in\mathbb{\mathbb{C}}.$ Combining (\ref{singular continuous linear Hamiltonian system}) and the boundary condition in $D$, the following result is straightforward.

\begin{lemma}
Assume that (\ref{singular continuous linear Hamiltonian system}) is in the limit circle cases at both endpoints $a$ and $b$. Then the spectrum of $H_D$ consists of isolated eigenvalues. Moreover, $\lambda\in\mathbb{C}$ is an eigenvalue of $H_D$ if and only of it is a zero of $\Gamma_c$.
\end{lemma}

Let $\lambda$ be an eigenvalue of $H_D$. The order of $\lambda$ as a zero of $\Gamma_c$ is called its analytic multiplicity. The number of linearly independent eigenfunctions for $\lambda$ is called its geometric multiplicity.

\subsection{Discrete case}

In this case, the maximal operator is not single-valued in general. To proceed it, we use the theory of linear relations and briefly introduce some related results.
Let $X$ be a Hilbert space with inner product $\langle\cdot,\cdot\rangle$, and $T$ be a self-adjoint  linear relation in $X^2$, which means $T=T^*:=\{(h,k)\in X^2:\langle g,h\rangle=\langle f,k\rangle$ for all $(f,g)\in T\}$. Denote $D(T)=\{f\in X:(f,g)\in T\textrm{\;for some }g\in X\}$ and $N(T)=\{f\in X:(f,0)\in T\}.$ Following \cite{Arens1961}, we decompose $T$ as
$T=T_s\oplus T_{\infty}$, where $T_\infty =\{(0,g)\in X^2:(0,g)\in T\}$. Then we know from Proposition 2.1 in \cite{Shi-Shao-Ren2013} that $T_\infty=T\cap T(0)^2$ and $T_s=T\cap (T(0)^{\bot})^2$. $T_s$ is a self-adjoint operator on $T(0)^{\bot}$ (see, P. 26 in \cite{Dijksma-Snoo1985}) and  $\sigma(T)=\sigma(T_s)$ (see Theorem 2.1 in \cite{Shi-Shao-Ren2013}). We define the product of two linear relations as $T_1T_2=\{(f,g)\in X^2:(f,h)\in T_2,(h,g)\in T_1\textrm{\;for some }h\in X\}$.
If $N(T-\lambda I)\neq0$, then $\lambda$ is called an eigenvalue of $T$. We define $\dim N(T-\lambda I)$ to be the geometric multiplicity of $\lambda$,  $\cup_{k\geq1}N((T-\lambda I)^k)$ to be the generalized eigenspace, and $\dim \cup_{k\geq1}N((T-\lambda I)^k)$ to be the algebraic multiplicity of $\lambda$. Then we give their relationship:
\begin{lemma}\label{linear relation multiplicity}
Let $T$ be a self-adjoint  linear relation in $X^2$ and $\lambda$ be an eigenvalue of $T$. Then
 $N((T-\lambda I)^i)=N((T_s-\lambda I)^i)$ and $N(T-\lambda I)=N((T-\lambda I)^i)$ for all $i\geq1$. Consequently, the algebraic and geometric multiplicities of $\lambda$ are equal.
\end{lemma}
\begin{proof} We give the proof by induction.
By Theorem 2.2 in \cite{Shi-Shao-Ren2013}, $N(T-\lambda I )= N(T_s-\lambda I )$. Suppose that $ N((T-\lambda I)^k)=N((T_s-\lambda I)^k)$ for some fixed $k\geq1$. We claim that $ N((T-\lambda I)^{k+1})=N((T_s-\lambda I)^{k+1})$. Indeed,  $N((T_s-\lambda I)^{k+1})\subset N((T-\lambda I)^{k+1})$ is evident. Conversely, let $f\in N((T-\lambda I)^{k+1})$.
Then $(f,0)\in (T-\lambda I)^{k}(T-\lambda I)$. So there exists $g\in X$ such that $(f,g)\in T-\lambda I$ and $(g,0)\in (T-\lambda I)^k$.
This gives $g\in N((T-\lambda I)^k)= N((T_s-\lambda I)^k)$ and $(f,g+\lambda f)\in T$.
Thus, $(g,0)\in (T_s-\lambda I)^k$, which, in particular, yields $g\in T(0)^{\perp}$.
We get by the decomposition of $T$ that $(f,g+\lambda f)=(f,g+\lambda f-h)+(0,h)$, where $(f,g+\lambda f-h)\in T_s$ and $(0,h)\in T_\infty$.
So $f,g+\lambda f-h\in T(0)^{\perp}$, which along with $g\in T(0)^{\perp}$, gives $h\in T(0)^{\perp}$. Noting that $h\in T(0)$, we obtain $h=0$.
Hence, $(f,g+\lambda f)\in T_s$ and $(f,g)\in T_s-\lambda I$. Recall that $(g,0)\in (T_s-\lambda I)^k$, we have $(f,0)\in(T_s-\lambda I)^{k+1}$, which gives $f\in N ((T_s-\lambda I)^{k+1})$.

Since $T_s$ is a self-adjoint operator, $N(T-\lambda I)=N(T_s-\lambda I)=N((T_s-\lambda I)^i)=N((T-\lambda I)^i)$ for all $i\geq1$.
 This completes the proof.\end{proof}

Define the space
\begin{align*}
\hat l^2_W(\mathbb{Z})=\left\{\{\hat f(n)\}_{n\in\mathbb{Z}}\subset\mathbb{C}^{2m}:\sum_{n\in\mathbb{Z}}R(\hat f)^*(n)W(n)R(\hat f)(n)<\infty\right\}
\end{align*}
with the semi-inner product $\langle \hat y,\hat z\rangle_d:=\sum_{n\in\mathbb{Z}}R(\hat z)^*(n)W(n)R(\hat y)(n)$ and let $l^2_W(\mathbb{Z})=\hat l^2_W(\mathbb{Z})/\{\hat y\in \hat l^2_W(\mathbb{Z}):\|\hat y\|_d=0\}.$ The maximal, pre-minimal and minimal linear relations  corresponding to (\ref{singular discrete linear Hamiltonian system}) are defined as follows:
\begin{align*}
\mathcal{T}=&\{(y,g)\in (l^2_W(\mathbb{Z}))^2:\textrm{there exists } \hat y\in y \textrm{ such that }
l (\hat y)(n)=W(n)R(\hat g)(n),n\in\mathbb{Z} \textrm{ for any}\\
&\;\hat g\in g\},\\
\mathcal{T}_{00}=&\{(y,g)\in \mathcal{T}:\textrm{ there  exist } \hat y\in y \textrm{ and two integers } s\leq k \textrm{ such that } \hat y(n)=0 \textrm{ for }n\leq s \textrm{ and}\\
 &  \;n\geq k+1,\textrm{ and } l(\hat y)(n)=W(n)R(\hat g)(n), \;n\in\mathbb{Z} \textrm{ for any }\hat g\in g\},\\
\mathcal{T}_{0}=&\overline{\mathcal{T}}_{00},
\end{align*}
where $l(\hat y)(n)=J\Delta \hat y(n)-P(n)R(\hat y)(n).$
We remark here that $\mathcal{T}$, $\mathcal{T}_{00}$ and $\mathcal{T}_{0}$ are not necessarily single-valued (see \cite{Ren-Shi2011,Ren-Shi2014}).  It was shown in \cite{Ren-Shi2011} that for any $(y,g)\in \mathcal{T}$, there exists a unique $\hat y\in y$ such that  $l(\hat y)(n)=W(n)R(\hat g)(n)$ for $n\in \mathbb{Z}$ by using $(B_1)$ and $(B_2)$. Thus we identify $\hat y$ and $y$ when $y\in D(\mathcal{T})$. Moreover,
\begin{align*}
(x,y)_d(n):=y^*(n)Jx(n),\;\;n\in \mathbb{Z}
\end{align*}
is well-defined for any $x,y\in D(\mathcal{T})$, and
\begin{align*}
\sum_{n=s}^k[R(y)^*(n)l(x)(n)-l(y)^*(n)R(x)(n)]=(x,y)_d(n)|_{s}^{k+1},
\end{align*}
which gives $(x,y)_d(\pm\infty):=\lim_{n\to\pm\infty}(x,y)_d(n)$ exists and are finite.

Let (\ref{singular discrete linear Hamiltonian system}) be in the limit circle cases at both  endpoints $\pm\infty$. Limit circle endpoint in discrete case is defined by defect indices of the left and right minimal linear relations, see Definition 5.1 in \cite{Shi2006} or Definition 3.1 in \cite{Ren-Shi2014}. Then we follow from Theorem 5.1 in \cite{Ren-Shi2011} that for all $\lambda\in\mathbb{C}$, (\ref{singular discrete linear Hamiltonian system}) has  $2m$ linearly independent solutions in $\hat l_W^2(\mathbb{Z})$, denoted by $\theta_{i,\lambda}$ ($1\leq i\leq 2m$), satisfying \begin{align*}
(\theta_{1,\lambda}(0),\theta_{2,\lambda}(0),\cdots,\theta_{2m,\lambda}(0))=I_{2m}.
\end{align*}
We define $\theta_i=\theta_{i,0}$, $1\leq i\leq 2m$,  and
\begin{align*}
\Theta_{\lambda}(n)=\begin{pmatrix}(\theta_{1,\lambda},\theta_1)_d(n)&\cdots&(\theta_{2m,\lambda},\theta_{1})_d(n)\\
\vdots&&\vdots\\(\theta_{1,\lambda},\theta_{2m})_d(n)&\cdots&(\theta_{2m,\lambda},\theta_{2m})_d(n)
\end{pmatrix},n\in\mathbb{Z}\cup\{\pm\infty\}.
\end{align*}
For a fixed $n\in\mathbb{Z}\cup\{\pm\infty\}$ and any $1\leq i,j\leq 2m$, $(\theta_{i,\lambda},\theta_j)_d(n)$ is an entire function of $\lambda$.
Next, the self-adjoint boundary conditions  are given in the following lemma:

\begin{lemma}\label{self-adjoint boundary condition in discrete case}
Assume that  $(\ref{singular discrete linear Hamiltonian system})$ is in the limit circle cases at $\pm\infty$.
Then a linear relation $\mathcal{T}_D\subset (l_W^2(\mathbb{Z}))^2$ is a self-adjoint linear relation extension of $\mathcal{T}_0$ if and only if there exist two $2m\times 2m$ matrices $M$ and $N$ such that $rank(M,N)=2m, MJM^*=NJN^*$, and
\begin{align}\label{boundary condition discrete}
\mathcal{T}_D=\left\{ (y,g)\in \mathcal{T}:M\begin{pmatrix}(y,\theta_1)_d(-\infty)\\\vdots\\(y,\theta_{2m})_d(-\infty)\end{pmatrix}
-N\begin{pmatrix}(y,\theta_1)_d(+\infty)\\\vdots\\(y,\theta_{2m})_d(+\infty)\end{pmatrix}=0\right\}.
\end{align}
\end{lemma}
We refer  to Theorem 5.6 in \cite{Ren-Shi2014} for the proof of Lemma  \ref{self-adjoint boundary condition in discrete case}.
 Let $\Gamma_d(\lambda)=\det(M\Theta_{\lambda}(-\infty)-N\Theta_{\lambda}(+\infty)), \lambda\in\mathbb{\mathbb{C}}.$ Then the spectrum of $\mathcal{T}_D$ can be  characterized as follows.

\begin{lemma}
Assume that (\ref{singular discrete linear Hamiltonian system}) is in the limit circle cases at $\pm\infty$. Then the spectrum of $\mathcal{T}_D$ consists of isolated eigenvalues. Moreover, $\lambda\in\mathbb{C}$ is an eigenvalue of $\mathcal{T}_D$ if and only of it is a zero of $\Gamma_d$.
\end{lemma}
 The analytic  multiplicity of an eigenvalue in discrete case is its order  as a zero of $\Gamma_d$.

\section{Continuous Hamiltonian system: equality of multiplicities of eigenvalues}

In this section, we give the relationship of analytic and geometric multiplicities of an eigenvalue of  the singular continuous linear Hamiltonian equation  (\ref{singular continuous linear Hamiltonian system}) with boundary condition in (\ref{boundary condition continuous}).

\begin{theorem}\label{continuous}
Assume that (\ref{singular continuous linear Hamiltonian system}) is in the limit circle cases at both endpoints $a$ and $b$.
Let $\lambda_0$ be an eigenvalue of $H_D$.
Then its analytic and geometric multiplicities are equal.
\end{theorem}

\begin{proof}
Denote the analytic and geometric multiplicities of $\lambda_0$ by $\tau_1$ and $\tau_2$, respectively. Let $\varphi_i$, $1\leq i\leq \tau_2$, be the linearly independent eigenfunctions for $\lambda_0$, and $\varphi_i$, $\tau_2+1\leq i\leq 2m$, be the solutions of (\ref{singular continuous linear Hamiltonian system}) with $\lambda=\lambda_0$ such that $\varphi_i$, $1\leq i\leq 2m$, are linearly independent.
Let $y_{i,\lambda}$ be the solutions of (\ref{singular continuous linear Hamiltonian system}) with $\lambda\in\mathbb{C}$ such that
$y_{i,\lambda}(c_0)=\varphi_i(c_0), 1\leq i\leq 2m$. Then $y_{i,\lambda_0}=\varphi_i$, $1\leq i\leq 2m$. For a fixed $t\in[a,b]$ and any $1\leq i,j,k\leq 2m$, recall that $(\phi_{k,\lambda},\phi_j)_c(t)$ is an entire function of $\lambda$, and thus we have the Taylor expansion $(y_{i,\lambda},\phi_j)_c(t)=\sum_{l=0}^\infty\varphi_{i,j,l}(t)(\lambda-\lambda_0)^l$ with $\varphi_{i,j,0}(t)=(\varphi_i,\phi_j)_c(t)$.
Define
\begin{align*}
\Psi_{\lambda}(t)=\begin{pmatrix}(y_{1,\lambda},\phi_1)_c(t)&\cdots&(y_{2m,\lambda},\phi_1)_c(t)\\
\vdots&&\vdots\\(y_{1,\lambda},\phi_{2m})_c(t)&\cdots&(y_{2m,\lambda},\phi_{2m})_c(t)\end{pmatrix},t\in [a,b].
\end{align*}
Then
\begin{align*}
 \Gamma_c(\lambda)=&\det(M\Phi_{\lambda}(a)-N\Phi_{\lambda}(b))\\
 =&\det\left(M\begin{pmatrix}\phi_{1}^*J\\\vdots\\\phi_{2m}^*J\end{pmatrix}(y_{1,\lambda}\cdots y_{2m,\lambda})(a)-N\begin{pmatrix}\phi_{1}^*J\\\vdots\\\phi_{2m}^*J\end{pmatrix}(y_{1,\lambda}\cdots y_{2m,\lambda})(b)\right)\cdot\\
 &\det\left((\varphi_1(c_0)\cdots\varphi_{2m}(c_0))^{-1}\right)\\
 =&\det(M\Psi_{\lambda}(a)-N\Psi_{\lambda}(b))\det\left((\varphi_1(c_0)\cdots \varphi_{2m}(c_0))^{-1}\right).
\end{align*}
Noting that
\begin{align*}
M\begin{pmatrix}(\varphi_i,\phi_1)_c(a)\\\vdots\\(\varphi_i,\phi_{2m})_c(a)\end{pmatrix}-
N\begin{pmatrix}(\varphi_i,\phi_1)_c(b)\\\vdots\\(\varphi_i,\phi_{2m})_c(b)\end{pmatrix}=0, \;\;1\leq i\leq \tau_2,
\end{align*}
we infer that the first $i$-th column  of $M\Psi_{\lambda}(a)-N\Psi_{\lambda}(b)$ contains the factor $\lambda-\lambda_0$.
Thus $\Gamma_c$ can be written as $\Gamma_c(\lambda)=(\lambda-\lambda_0)^{\tau_2}\tilde\Gamma_c(\lambda), \lambda\in\mathbb{C}$, where $\tilde \Gamma_c$ is an entire function of $\lambda$. Then we claim  that  $\tilde\Gamma_c(\lambda_0)\neq0$. This gives $\tau_1=\tau_2$ and thus completes the proof.

It remains to prove that $\tilde \Gamma_c(\lambda_0)\neq0$.
Let
\begin{align*}
\tilde \Phi(t)=\begin{pmatrix}\varphi_{1,1,1}(t)&\cdots&\varphi_{\tau_2,1,1}(t)&(\varphi_{\tau_2+1},\phi_1)_c(t)
&\cdots&(\varphi_{2m},\phi_1)_c(t)\\\vdots&&\vdots&\vdots&&\vdots\\
\varphi_{1,2m,1}(t)&\cdots&\varphi_{\tau_2,2m,1}(t)&(\varphi_{\tau_2+1},\phi_{2m})_c(t)
&\cdots&(\varphi_{2m},\phi_{2m})_c(t)\end{pmatrix},t\in[a,b].
\end{align*}
Then $\tilde \Gamma_c(\lambda_0)=\det(M\tilde\Phi(a)-N\tilde\Phi(b))\det\left((\varphi_1(c_0)\cdots \varphi_{2m}(c_0))^{-1}\right).$

Suppose that $\tilde \Gamma_c(\lambda_0)=0$. Then  there exist $c_i\in\mathbb{C}, 1\leq i\leq 2m$,  such that they are not all zeros and  $\sum_{i=1}^{2m}c_i\xi_i=0$, where $\xi_i, 1\leq i\leq 2m$, are the columns of $M\tilde\Phi(a)-N\tilde\Phi(b)$.

If $c_i=0$ for all $1\leq i\leq \tau_2$, then we get $\psi=\sum_{i=\tau_2+1}^{2m}c_i\varphi_i$ is a nontrivial solution of $(\ref{singular continuous linear Hamiltonian system})$ with $\lambda=\lambda_0$. Since $\sum_{i=\tau_2+1}^{2m}c_i\xi_i=0$, we have
\begin{align*}
M\begin{pmatrix}(\psi,\phi_1)_c(a)\\\vdots\\(\psi,\phi_{2m})_c(a)\end{pmatrix}-
N\begin{pmatrix}(\psi,\phi_1)_c(b)\\\vdots\\(\psi,\phi_{2m})_c(b)\end{pmatrix}=0.
\end{align*}
Thus $\psi\in\hat L^2_W(a,b)$ is an eigenfunction for $\lambda_0$. On the other hand, there exist $\tilde c_i\in\mathbb{C}$, $1\leq i\leq \tau_2$,  such that they are not all zeros and $\psi=\sum_{i=\tau_2+1}^{2m}c_i\varphi_i=\sum_{i=1}^{\tau_2}\tilde c_i\varphi_i$, which contradicts the linear independence of $\varphi_i$, $1\leq i\leq 2m$.

If $c_i, 1\leq i\leq \tau_2$, are not all vanished, then for any $\lambda\in\mathbb{C}$,
\begin{align}\label{Hamiltonian-equation-z}
Jz'(t,\lambda)=(P(t)+\lambda W(t))z(t,\lambda),\;t\in(a,b),
\end{align}
 where  $z(\cdot,\lambda)=\sum_{i=1}^{\tau_2}c_iy_{i,\lambda}(\cdot)+\sum_{i=\tau_2+1}^{2m}c_i(\lambda-\lambda_0)y_{i,\lambda}(\cdot)\in \hat L_W^2(a,b)$ is nontrivial for any fixed $\lambda\in\mathbb{C}$. Furthermore, $z(\cdot,\lambda_0)$ satisfies the boundary condition in (\ref{boundary condition continuous}).
 Differentiating (\ref{Hamiltonian-equation-z}) by $\lambda$, we have
 \begin{align*}
J\left({\partial_\lambda z}\right)'(t,\cdot)-P(t){\partial_\lambda z}(t,\cdot)= W(t)\left(z(t,\cdot)+\lambda {\partial_\lambda z}(t,\cdot)\right),\;t\in(a,b).
\end{align*}
 This yields that ${\partial_\lambda z}(\cdot,\lambda_0)$ is nontrivial, since otherwise, $W(t)z(t,\lambda_0)=0$ on $(a,b)$, which contradicts $z(\cdot,\lambda_0)$ is a nontrivial solution of (\ref{singular continuous linear Hamiltonian system}) with $\lambda=\lambda_0$ by (A).
Moreover, we have
\begin{align*}
&\int_a^b\partial_\lambda z(t,\lambda_0)^*W(t)\partial_\lambda z(t,\lambda_0)dt={1\over4\pi^2}\int_a^b\left(\int_{\gamma_\rho}{z(t,\lambda)\over(\lambda-\lambda_0)^2}d\lambda\right)^*W(t)
\left(\int_{\gamma_\rho}{z(t,\lambda)\over(\lambda-\lambda_0)^2}d\lambda\right)dt\\
=&{1\over4\pi^2}\int_a^b\left(\int_0^{2\pi}{z(t,\lambda_0+\rho e^{i\beta})i\rho e^{i\beta}\over\rho^2 e^{2i\beta}}d\beta\right)^*W(t)\left(\int_0^{2\pi}{z(t,\lambda_0+\rho e^{i\beta})i\rho e^{i\beta}\over\rho^2 e^{2i\beta}}d\beta\right)dt\\
=&{1\over4\pi^2}\int_0^{2\pi}\int_0^{2\pi}{1\over\rho^2e^{i(\beta_2-\beta_1)}}\langle z(\cdot,\lambda_0+\rho e^{i\beta_2}),z(\cdot,\lambda_0+\rho e^{i\beta_1})\rangle_c d\beta_1d\beta_2\\
\leq&{1\over\rho^2}\max_{\lambda\in\gamma_\rho}\|z(\cdot,\lambda)\|^2_c
\leq C({1\over\rho^2}+1)e^{C\rho}\max_{\lambda\in\gamma_\rho}\|Z(c_0,\lambda)\|_{2m}^2\max_{1\leq i\leq 2m}\|\phi_{i,\lambda_0}\|_c^2,
\end{align*}
where $C$ is a generic constant, $\gamma_\rho$ is a circle centred at $\lambda_0$ with radius $\rho>0$, $\|Z\|_{2m}=(\sum_{i=1}^{2m}\sum_{j=1}^{2m}|z_{ij}|^2)^{1\over2}$ for a matrix $Z$ with entries $z_{ij}$, and $(\phi_{1,\lambda},\cdots,\phi_{2m,\lambda})=(\phi_{1,\lambda_0},\cdots,\phi_{2m,\lambda_0})\cdot$ $Z(\cdot,\lambda)$. We refer the readers to the proof of Theorem 9.11.2 in \cite{Atkinson1964} for the last inequality above. This implies $\partial_\lambda z(\cdot,\lambda_0)\in \hat L_W^2(a,b)$.
Noticing  that $(\partial_\lambda z(\cdot,\lambda_0),\phi_j)_c=\sum_{i=1}^{\tau_2}c_i\varphi_{i,j,1}+\sum_{i=\tau_2+1}^{2m}c_i(\varphi_i,\phi_j)_c$
and $\sum_{i=1}^{2m}c_i\xi_i=0$, we get
\begin{align*}
M\begin{pmatrix}(\partial_\lambda z(\cdot,\lambda_0),\phi_1)_c(a)\\\vdots\\(\partial_\lambda z(\cdot,\lambda_0),\phi_{2m})_c(a)\end{pmatrix}-
N\begin{pmatrix}(\partial_\lambda z(\cdot,\lambda_0),\phi_1)_c(b)\\\vdots\\(\partial_\lambda z(\cdot,\lambda_0),\phi_{2m})_c(b)\end{pmatrix}=0.
\end{align*}
Hence, $\partial_\lambda z(\cdot,\lambda_0)$ is a generalized eigenfunction for $\lambda_0$, which contradicts that $H_D$ is self-adjoint.
 This completes the proof.\end{proof}

\section{Discrete Hamiltonian system: equality of multiplicities of eigenvalues}

In this section, we obtain the  relationship of  multiplicities of an eigenvalue of the singular discrete linear  Hamiltonian equation  (\ref{singular discrete linear Hamiltonian system}) with boundary condition in (\ref{boundary condition discrete}).

\begin{theorem}\label{discrete}
Assume that  $(\ref{singular discrete linear Hamiltonian system})$ is in the limit circle cases at $\pm\infty$.
Let $\lambda_*$ be an eigenvalue of   $\mathcal{T}_D$.
Then its analytic and geometric multiplicities are equal.
\end{theorem}

\begin{proof}
Denote the analytic and geometric multiplicities of $\lambda_*$ by $\mu_1$ and $\mu_2$, respectively. Let $\zeta_i$, $1\leq i\leq \mu_2$, be the linearly independent eigenfunctions for $\lambda_*$, and $\zeta_i$, $\mu_2+1\leq i\leq 2m$, be the solutions of (\ref{singular discrete linear Hamiltonian system}) with $\lambda=\lambda_*$ such that $\zeta_i$, $1\leq i\leq 2m$, are linearly independent.
Set $\chi_{i,\lambda}$ be the solutions of (\ref{singular discrete linear Hamiltonian system}) with $\lambda\in\mathbb{C}$ such that
$\chi_{i,\lambda}(0)=\zeta_i(0), 1\leq i\leq 2m$, and then $\chi_{i,\lambda_*}=\zeta_i$. For a fixed $n\in\mathbb{Z}\cup\{\pm\infty\}$ and any $1\leq i,j\leq 2m$, we have  $(\chi_{i,\lambda},\theta_j)_d(n)=\sum_{l=0}^\infty\zeta_{i,j,l}(n)(\lambda-\lambda_*)^l$ with $\zeta_{i,j,0}(n)=(\zeta_i,\theta_j)_d(n)$.
Let
\begin{align*}
\Omega_{\lambda}(n)=\begin{pmatrix}(\chi_{1,\lambda},\theta_1)_d(n)&\cdots&(\chi_{2m,\lambda},\theta_1)_d(n)\\
\vdots&&\vdots\\(\chi_{1,\lambda},\theta_{2m})_d(n)&\cdots&(\chi_{2m,\lambda},\theta_{2m})_d(n)\end{pmatrix},\; n\in\mathbb{Z}\cup\{\pm\infty\}.
\end{align*}
By the definition of $\Gamma_d$, we obtain
\begin{align*}
 &\Gamma_d(\lambda)=\det(M\Theta_{\lambda}(-\infty)-N\Theta_{\lambda}(+\infty))\\
 =&\det(M\Omega_{\lambda}(-\infty)-N\Omega_{\lambda}(+\infty))\det\left((\zeta_1(0)\cdots \zeta_{2m}(0))^{-1}\right).
\end{align*}
Since
\begin{align*}
M\begin{pmatrix}(\zeta_i,\theta_1)_d(-\infty)\\\vdots\\(\zeta_i,\theta_{2m})_d(-\infty)\end{pmatrix}-
N\begin{pmatrix}(\zeta_i,\theta_1)_d(+\infty)\\\vdots\\(\zeta_i,\theta_{2m})_d(+\infty)\end{pmatrix}=0, \;\;1\leq i\leq \mu_2,
\end{align*}
 the first $i$-th column  of $M\Omega_{\lambda}(-\infty)-N\Omega_{\lambda}(+\infty)$ contains the factor $\lambda-\lambda_*$.
Thus $\Gamma_d(\lambda)=(\lambda-\lambda_*)^{\mu_2}\tilde\Gamma_d(\lambda), \lambda\in\mathbb{C}$, where $\tilde \Gamma_d$ is an entire function of $\lambda$. Then we prove that  $\tilde\Gamma_d(\lambda_*)\neq0$. Define
\begin{align*}
\tilde \Theta(n)=\begin{pmatrix}\zeta_{1,1,1}(n)&\cdots&\zeta_{\mu_2,1,1}(n)&(\zeta_{\mu_2+1},\theta_1)_d(n)
&\cdots&(\zeta_{2m},\theta_1)_d(n)\\\vdots&&\vdots&\vdots&&\vdots\\
\zeta_{1,2m,1}(n)&\cdots&\zeta_{\mu_2,2m,1}(n)&(\zeta_{\mu_2+1},\theta_{2m})_d(n)
&\cdots&(\zeta_{2m},\theta_{2m})_d(n)\end{pmatrix},
\end{align*}
where $n\in\mathbb{Z}\cup\{\pm\infty\}$. Then $\tilde \Gamma_d(\lambda_*)=\det(M\tilde\Theta(-\infty)-N\tilde\Theta(+\infty))\det\left((\zeta_1(0)\cdots \zeta_{2m}(0))^{-1}\right).$

Suppose that $\tilde \Gamma_d(\lambda_*)=0$. Then  there exist $d_i\in\mathbb{C}, 1\leq i\leq 2m$,  such that they are not all zeros and  $\sum_{i=1}^{2m}d_i\eta_i=0$, where $\eta_i, 1\leq i\leq 2m$, are the columns of $M\tilde\Theta(-\infty)-N\tilde\Theta(+\infty)$.

If $d_i=0$ for all $1\leq i\leq \mu_2$, then we get $\tilde\psi=\sum_{i=\mu_2+1}^{2m}d_i\zeta_i$ is a nontrivial solution of $(\ref{singular discrete linear Hamiltonian system})$ with $\lambda=\lambda_*$ and $\sum_{i=\mu_2+1}^{2m}d_i\eta_i=0$. This gives
\begin{align*}
M\begin{pmatrix}(\tilde\psi,\theta_1)_d(-\infty)\\\vdots\\(\tilde\psi,\theta_{2m})_d(-\infty)\end{pmatrix}-
N\begin{pmatrix}(\tilde\psi,\theta_1)_d(+\infty)\\\vdots\\(\tilde\psi,\theta_{2m})_d(+\infty)\end{pmatrix}=0.
\end{align*}
Thus $\tilde\psi\in\hat l^2_W(\mathbb{Z})$ is an eigenfunction for $\lambda_*$ and there exist $\tilde d_i\in\mathbb{C}$, $1\leq i\leq \mu_2$, such that they are not all zeros and $\tilde\psi=\sum_{i=\mu_2+1}^{2m}d_i\zeta_i=\sum_{i=1}^{\mu_2}\tilde d_i\zeta_i$, which contradicts the linear independence of $\zeta_i$, $1\leq i\leq 2m$.

If $d_i, 1\leq i\leq \mu_2$, are not all zeros, then for any $\lambda\in\mathbb{C}$,
\begin{align}\label{Hamiltonian-equation-x}
J\Delta x(n,\lambda)=(P(n)+\lambda W(n))R(x)(n,\lambda),\;n\in\mathbb{Z},
\end{align}
 where  $x(\cdot,\lambda)=\sum_{i=1}^{\mu_2}d_i\chi_{i,\lambda}(\cdot)+\sum_{i=\mu_2+1}^{2m}d_i(\lambda-\lambda_*)\chi_{i,\lambda}(\cdot)\in \hat l_W^2(\mathbb{Z})$ is nontrivial for any fixed $\lambda\in\mathbb{C}$. In addition, $x$ satisfies the boundary condition in (\ref{boundary condition discrete}).
 We get by (\ref{Hamiltonian-equation-x}) that
 \begin{align*}
J\Delta\left({\partial_\lambda x}\right)(n,\cdot)-P(n)R({\partial_\lambda x})(n,\cdot)= W(n)R(\lambda\partial_\lambda x+x)(n,\cdot),\;n\in \mathbb{Z}.
\end{align*}
 Then ${\partial_\lambda x}(\cdot,\lambda_*)$ is nontrivial, since otherwise, $W(n)R(x)(n,\lambda_*)\equiv0$ on $\mathbb{Z}$, which contradicts
 $x(\cdot,\lambda_*)$ is nontrivial according to
(\ref{Hamiltonian-equation-x}) and ($B_2$).
Direct computation gives
\begin{align*}
&\sum_{n\in\mathbb{Z}}R(\partial_\lambda x)(n,\lambda_*)^*W(n)R(\partial_\lambda x)(n,\lambda_*)\\
=&{1\over4\pi^2}\sum_{n\in\mathbb{Z}}R\left(\int_{\gamma_\rho}{x(n,\lambda)\over(\lambda-\lambda_*)^2}d\lambda\right)^*W(n)
R\left(\int_{\gamma_\rho}{x(n,\lambda)\over(\lambda-\lambda_*)^2}d\lambda\right)\\
=&{1\over4\pi^2}\sum_{n\in\mathbb{Z}}\left(\int_0^{2\pi}{R(x)(n,\lambda_*+\rho e^{i\beta})\over\rho e^{i\beta}}d\beta\right)^*W(n)\left(\int_0^{2\pi}{R(x)(n,\lambda_*+\rho e^{i\beta})\over\rho e^{i\beta}}d\beta\right)\\
=&{1\over4\pi^2}\int_0^{2\pi}\int_0^{2\pi}{1\over\rho^2e^{i(\beta_2-\beta_1)}}\langle x(\cdot,\lambda_*+\rho e^{i\beta_2}),x(\cdot,\lambda_*+\rho e^{i\beta_1})\rangle_d d\beta_1d\beta_2\\
\leq&{1\over\rho^2}\max_{\lambda\in\gamma_\rho}\|x(\cdot,\lambda)\|_d^2\leq C({1\over\rho^2}+1)e^{C\rho}\max_{\lambda\in\gamma_\rho}\|X(0,\lambda)\|_{2m}^2\max_{1\leq i\leq 2m}\|\theta_{i,\lambda_*}\|_d^2,
\end{align*}
where $\gamma_\rho$ is a circle centred at $\lambda_*$ with radius $\rho$, and $(\theta_{1,\lambda},\cdots,\theta_{2m,\lambda})=(\theta_{1,\lambda_*},\cdots,\theta_{2m,\lambda_*})\cdot$ $X(\cdot,\lambda)$. The last inequality is due to the proof of Theorem 5.5 in \cite{Shi2006}.
This gives $\partial_\lambda x(\cdot,\lambda_*)\in \hat l_W^2(\mathbb{Z})$.
Since $(\partial_\lambda x(\cdot,\lambda_*),\theta_j)_d=\sum_{i=1}^{\mu_2}d_i\zeta_{i,j,1}+\sum_{i=\mu_2+1}^{2m}d_i(\zeta_i,\theta_j)_d$
and $\tilde \Gamma_d(\lambda_*)=0$, we have
\begin{align*}
M\begin{pmatrix}(\partial_\lambda x(\cdot,\lambda_*),\theta_1)_d(-\infty)\\\vdots\\(\partial_\lambda x(\cdot,\lambda_*),\theta_{2m})_d(-\infty)\end{pmatrix}-
N\begin{pmatrix}(\partial_\lambda x(\cdot,\lambda_*),\theta_1)_d(+\infty)\\\vdots\\(\partial_\lambda x(\cdot,\lambda_*),\theta_{2m})_d(+\infty)\end{pmatrix}=0.
\end{align*}
Thus, we conclude that $0\neq\partial_\lambda x(\cdot,\lambda_*)\in N((\mathcal{T}_D-\lambda_* I)^2)\setminus N(\mathcal{T}_D-\lambda_* I)$ is a generalized eigenfunction for $\lambda_*$. On the other hand, since $\mathcal{T}_D$ is a self-adjoint linear relation, we get by Lemma \ref{linear relation multiplicity} that $N(\mathcal{T}_D-\lambda_* I)=N((\mathcal{T}_D-\lambda_* I)^2)$, which is a contradiction.

Hence, $\mu_1=\mu_2$ and this completes the proof.\end{proof}

\section*{Acknowledgement}
 H. Zhu is partially supported  by  PITSP (No. BX20180151) and CPSF (No. 2018M630266).

\end{CJK*}


\begin{thebibliography}{99}

\bibitem {Arens1961} R. Arens, Operational calculus of linear realtions, Pacific J. Math., 11 (1961), 9--23.
\bibitem {Atkinson1964} F.V. Atlinson, Discrete and Continuous Boundary Problems, Academic Press, Inc., New York, 1964.

\bibitem {Dijksma-Snoo1985} A. Dijksma, H.S.V. de Snoo, Eigenfunction expansions associated with pairs of ordinary differential expressions, J. Differential Equations, 60 (1985), 21--56.
  \bibitem {Eastham1999} M. S. P. Eastham, Q. Kong, H. Wu, A. Zettl, Inequalities among eigenvalues of Sturm-Liouville problems,
J. Inequalities and Appl., 3 (1999), 25--43.
\bibitem{Hu-Liu-Wu-Zhu2018}  X. Hu, L. Liu, L. Wu, H. Zhu, Singularity of the $n$-th eigenvalue of  high dimensional Sturm-Liouville problems, arXiv:1805.00253.
\bibitem{Hu-Wang2011}   X. Hu, P. Wang, Conditional Fredholm determinant for the S-periodic orbits in Hamiltonian systems, J. Funct. Anal.,  261  (2011),  3247--3278.
\bibitem{Hu-Wang20161}   X. Hu, P. Wang, Hill-type formula and Krein-type trace formula for S-periodic solutions in ODEs, Discrete Contin. Dyn. Syst.,  36  (2016),   763--784.
 \bibitem{Hu-Wang20162}   X. Hu, P. Wang,  Eigenvalue problem of Sturm-Liouville systems with separated boundary conditions, Math. Z.,  283  (2016), 339--348.
  \bibitem{Kong2000} Q. Kong, H. Wu, A. Zettl, Geometric aspects of Sturm-Liouville
problems, I. Structures on spaces of boundary conditions,
Proc. Roy. Soc. Edinburgh Sect. A, 130 (2000), 561--589.

\bibitem{Kong2004} Q. Kong, H. Wu, A. Zettl, Multiplicity of Sturm-Liouville eigenvalues,
J. Comput. Appl. Math., 171 (2004), 291--309.
\bibitem{Naimark1968}
M. A. Naimark, Linear Differential Operators, Ungar, New York, 1968.
\bibitem {Ren-Shi2011} G. Ren, Y. Shi, Defect indices and definiteness conditions for a class of discrete linear Hamiltonian systems, Appl. Math. Comput., 218 (2011), 3414--3429.
\bibitem {Ren-Shi2014} G. Ren, Y. Shi, Self-adjoint extensions for discrete linear Hamiltonian systems, Linear Algebra Appl., 454 (2014), 1--48.
\bibitem{Shi2010}
D. Shi, Z. Huang, Relationships of multiplicities of eigenvalues of a higher-order ordinary differential operator,
Acta Math. Sinica (Chin. Ser.),  53  (2010), 763--772.
\bibitem {Shi2004} Y. Shi, On the rank of the matrix radius of the limiting set for a singular linear Hamiltonian system, Linear Algebra Appl.,  376  (2004), 109--123.
 \bibitem {Shi2006} Y. Shi, Weyl-Titchmarsh theory for a class of discrete linear Hamiltonian systems,  Linear Algebra Appl.,  416  (2006), 452--519.
\bibitem {Shi-Shao-Ren2013} Y. Shi, C. Shao, G. Ren,     Spectral properties of self-adjoint subspaces, Linear Algebra Appl., 483 (2013), 191--218.
\bibitem {Sun-Shi2010} H. Sun, Y. Shi, Self-adjoint extensions for linear Hamiltonian systems with two singular endpoints, J. Func. Anal., 259 (2010), 2003--2027.
 \bibitem{Wang2005}
 Z. Wang, H. Wu, Equalities of multiplicities of a Sturm-Liouville eigenvalue, J. Math. Anal. Appl.,  306  (2005), 540--547.
\bibitem{Zhu2016}
 H. Zhu, S. Sun, Y. Shi, H. Wu, Dependence of eigenvalues of certain closely discrete Sturm-Liouville problems, Complex Anal. Oper. Theory, 10 (2016), 667--702.

\end{thebibliography}
\end{document}